\documentclass{article}
\usepackage[english]{babel}
\usepackage[utf8]{inputenc}
\usepackage[T1]{fontenc}
\usepackage[toc,page]{appendix}
\usepackage{graphicx}
\usepackage{geometry}
\usepackage{sectsty}
\usepackage{bbm}
\usepackage{dsfont}
\usepackage{alltt}
\usepackage{verbatimbox}
\usepackage{amsthm,amsfonts}
\usepackage{amsmath}
\usepackage{mathabx}
\usepackage{accents}
\usepackage{titlesec}
\usepackage{mathtools}
\usepackage{empheq}
\mathtoolsset{showonlyrefs=true}

\titleformat*{\section}{\Large\bfseries}
\titleformat*{\subsection}{\large\bfseries}
\titleformat*{\subsubsection}{\large\bfseries}
\titleformat*{\paragraph}{\large\bfseries}
\titleformat*{\subparagraph}{\large\bfseries}

\newtheorem{teo}{Theorem}[section]
\newtheorem{lema}[teo]{Lemma}

\newtheorem{prop}[teo]{Proposition}
\newtheorem{defi}[teo]{Definition}
\newtheoremstyle{mytheoremstyle} 
{\topsep}                    
{\topsep}                    
{}                   
{}                           
{\scshape}                   
{.}                          
{.5em}                       
{}  

\theoremstyle{mytheoremstyle} \newtheorem{nota}{Remark}[section]
\theoremstyle{mytheoremstyle} 
\theoremstyle{mytheoremstyle} 
\numberwithin{equation}{section}
\newcommand{\real}{\mathbb{R}}

\newcommand{\nat}{\mathbb{N}}

\newcommand \ben {\begin{equation}}
\newcommand \een {\end{equation}}
\newcommand \be {\begin{equation*}}
\newcommand \ee {\end{equation*}}
\newcommand \bi {\begin{itemize}}
\newcommand \ei {\end{itemize}}

\DeclareMathOperator*{\parteim}{Im}

\title{\textbf{Local Cauchy theory for the nonlinear Schrödinger equation in spaces of infinite mass}}
\author{Simão Correia}
\begin{document}

\maketitle

\begin{abstract}
We consider the Cauchy problem for the nonlinear Schrödinger equation on $\real^d$, where the initial data is in $\dot{H}^1(\real^d)\cap L^p(\real^d)$. We prove local well-posedness for large ranges of $p$ and discuss some global well-posedness results.
\end{abstract}
\section{Introduction}
In this work, we consider the classical nonlinear Schrödinger equation over $\real^d$:
\begin{equation}\tag{NLS}
iu_t + \Delta u + \lambda |u|^\sigma u = 0, \quad u=u(t,x),\quad (t,x)\in \real\times\real^d, \ \lambda\in \real, \ 0<\sigma<4/(d-2)^+
\end{equation}
and focus on the corresponding Cauchy problem $u(0)=u_0\in E$, where $E$ is a suitable function space. This model equation is the subject of more than fifty years of intensive research, which makes us unable to give a complete list of important references (we simply refer the monographs \cite{cazenave}, \cite{sulem}, \cite{tao} and references therein). The usual framework one considers is $E=H^1(\real^d)$, the so-called energy space, or more generally, $E=H^s(\real^d)$. A common property of these spaces is that they are $L^2$-based. The reason for this constraint comes from the fact that the linear group is bounded in $L^2$, but not in any other $L^p$.

In the sense of lifting the $L^2$ constraint, we refer the papers \cite{gallo}, \cite{gerard} and \cite{ms2}. In the first paper, one considers local well-posedness on Zhidkov spaces
$$
E=X^k(\real^d)=\{u\in L^\infty(\real^d): \nabla u\in H^{k-1}(\real^d) \}.
$$
In the second, one takes the Gross-Pitaevskii equation and looks for local well-posedness on
$$
E=\{u\in H^1_{loc}(\real^d): \nabla u\in L^2(\real^d), |u|^2-1\in L^2(\real^d)\}.
$$
Finally, in the third work, one considers $E=H^1(\real^2)+X$, where $X$ is either a particular space of bounded functions with no decay or a subspace of $L^4(\real^2)$ (and not of $L^2(\real^2)$). 

The aim of this paper is to look for local well-posedness results over another class of spaces, namely
$$
E=X_p(\real^d)=\dot{H}^1(\real^d) \cap L^p(\real^d), \quad 2<p<2d/(d-2)^+.
$$
In particular, we obtain local well-posedness in the most general energy space $X_{\sigma+2}(\real^d)$ and obtain global well-posedness over $X_p(\real^d)$ in the defocusing case $\lambda<0$ for all $p\le \sigma+2$.

\begin{nota}
Our results can be extended to more general nonlinearities $f(u)$ as in the $H^1$ framework. We present our results for $f(u)=|u|^{\sigma}u$ so not to complicate unnecessarily the proofs and deviate from the main ideas.
\end{nota}

We briefly explain the structure of this work: in Section 2, we derive the required group estimates and show that the Schrödinger group is well-defined over $X_p(\real^d)$. In Section 3, we show local well-posedness for $p\le 2\sigma+2$, where the use of Strichartz estimates is available. We also prove global well-posedness for small $\sigma$ (cf. Proposition \ref{prop:globalsubcritico}). In Section 4, we deal with the complementary case $p>2\sigma+2$ in dimensions $d=1,2$.
\vskip15pt
\noindent \textbf{Notation.} The norm over $L^p(\real^d)$ will be denoted as $\|\cdot\|_p$ or $\|\cdot\|_{L^p}$, whichever is more convenient. The spatial domain $\real^d$ will often be ommited. The free Schrödinger group in $H^1(\real^d)$ is written as $\{S(t)\}_{t\in \real}$. We write $p^*=dp/(d-p)^+$. To avoid repetition, we hereby set $2<p<2^*$ and $0<\sigma<4/(d-2)^+$.

\section{Linear estimates}
We recall the essential Strichartz estimates. We say that $(q,r)$ is an admissible pair if
$$
2\le r \le 2^*,\quad \frac{2}{q}=d\left(\frac{1}{2}-\frac{1}{r}\right), \quad r\neq \infty\mbox{ if }d=2.
$$
\begin{lema}[Strichartz estimates]
Given two admissible pairs $(q,r)$ and $(\gamma, \rho)$, we have, for all sufficiently regular $u_0$ and $f$ and for any interval $I\subset \real$,
\begin{equation}\label{eq:strichartzhomogeneo}
\|S(\cdot)u_0\|_{L^q(I; L^r(\real^N))}\lesssim \|u_0\|_2
\end{equation}
and
\begin{equation}\label{eq:strichartznaohomogeneo}
\left\| \int_{0<s<t} S(t-s)f(s) ds\right\|_{L^q(I; L^r(\real^N))}\lesssim \|f\|_{L^{\gamma'}(I; L^{\rho'}(\real^N))}.
\end{equation}
\end{lema}
\begin{nota}
The estimate \eqref{eq:strichartznaohomogeneo} may be extended to other sets of admisssible pairs: see \cite{foschi} and \cite{vilela}. However, the linear estimate \eqref{eq:strichartzhomogeneo} is not valid for any other pairs and for $u_0\notin L^2(\real^d)$.
\end{nota}
\begin{prop}[Group estimates with loss of derivative]\label{prop:estimativasLp}
Define $k$ so that $(k,p)$ is admissible. Then
\begin{itemize}
\item (Linear estimate) For $\phi\in \mathcal{S}(\real^d)$,
\begin{equation}\label{eq:linear}
\|S(t)\phi\|_p^2\lesssim \|\phi\|_p^2 + |t|^{1-\frac{2}{k}}\|\nabla \phi\|_2^2,\quad t\in\real.
\end{equation}
\item (Non-homogeneous estimate) For $f\in C([0,T]; \mathcal{S}(\real^d))$ and any $(q,r)$ admissible, 
\begin{equation}\label{eq:naohomogenea}
\left\|\int_0^\cdot S(\cdot-s) f(s) ds\right\|_{L^\infty((0,T); L^p(\real^d))}\lesssim C(T)\left( \|f\|_{L^2((0,T); L^p(\real^d))} + \|\nabla f\|_{L^{q'}((0,T); L^{r'}(\real^d))}\right),
\end{equation}
where $C(\cdot)$ is a increasing bounded function over bounded intervals of $\real$.
\end{itemize}
\end{prop}

Notice that, due to the scaling invariance of the Schrödinger equation, the polynomial growth in time in the linear estimate is unavoidable.

\begin{proof}
For the linear estimate, write $u=S(t)\phi$. Then $u\in C^1(\real; H^2(\real^d))$ satisfies
$$
iu_t + \Delta u=0,\quad u(0)=\phi.
$$
Multiplying the equation by $|u|^{p-2}\bar{u}$, integrating over $\real^d$ and taking the imaginary part, we obtain
$$
\frac{1}{p}\frac{d}{dt} \|u(t)\|_p^p \le \left|\parteim \int |u|^{p-2}\bar{u}\Delta u\right|\le \frac{p-2}{2} \int |u|^{p-2}|\nabla u|^2 \le \frac{p-2}{2}\|u(t)\|_{p}^{p-2}\|\nabla u(t)\|_p^2
$$
Thus we have
$$
\frac{d}{dt} \|u(t)\|_p^2 \le (p-2)\|\nabla u(t)\|_p^2.
$$
An integration between $0$ and $t\in\real$ and the linear Strichartz estimate yield
$$
\|u(t)\|_p^2\lesssim \|\phi\|_p^2 + \int_0^t \|\nabla u(s)\|_p^2 ds\lesssim \|\phi\|_p^2 + |t|^{1-\frac{2}{k}}\left(\int_0^t \|\nabla u(s)\|_p^k ds\right)^{\frac{2}{k}} \lesssim \|\phi\|_p^2 + |t|^{1-\frac{2}{k}}\|\nabla \phi\|_2^2.
$$
For the non-homogeneous estimate, set $v(t)=-i\int_0^t S(t-s)f(s)ds$. Then $v\in C^1([0,T]; H^1(\real^d))$ satisfies
$$
iv_t+\Delta v = f,\quad v(0)=0.
$$
As for the previous estimate, we have
$$
\frac{1}{p}\frac{d}{dt} \|v(t)\|_p^p \lesssim\|v(t)\|_{p}^{p-2}\|\nabla v(t)\|_p^2 + \|v(t)\|_p^{p-1}\|f(t)\|_p
$$
and so
$$
\frac{d}{dt} \|v(t)\|_p^2 \lesssim \|\nabla v(t)\|_p^2 + \|v(t)\|_p\|f(t)\|_p \lesssim \|\nabla v(t)\|_p^2 + \|v(t)\|_p^2 + \|f(t)\|_p^2.
$$
The required estimate now follows by direct integration in $(0,t)$, $0<t<T$, and by the non-homogeneous Strichartz estimate.
\end{proof}

\begin{lema}[Local Strichartz estimate without loss of derivatives]
Given $f\in C([0,T], \mathcal{S}(\real^d))$,
\begin{equation}\label{strichartzlocal}
\left\|\int_0^\cdot S(\cdot-s)f(s)ds\right\|_{L^\infty((0,T), L^p(\real^d))} \lesssim C(T,q)\|f\|_{L^{q'}((0,T), L^{p'}(\real^d))}, \quad \frac{1}{q}>d\left(\frac{1}{2}-\frac{1}{p}\right).
\end{equation}
\end{lema}
\begin{proof}
This estimate follows easily from the decay estimates of the Schrödinger group: indeed, given $0<t<T$,
\begin{align*}
\left\|\int_0^t S(t-s)f(s)ds\right\|_{L^p(\real^d)}&\lesssim \int_0^T \|S(t-s)f(s)\|_{L^p(\real^d)}ds \\&\lesssim \int_0^T \frac{1}{|t-s|^{d\left(\frac{1}{2}-\frac{1}{p}\right)}}\|f(s)\|_{L^{p'}(\real^d)}ds \\&\lesssim \left(\int_0^T \frac{1}{|t-s|^{qd\left(\frac{1}{2}-\frac{1}{p}\right)}} ds\right)^{\frac{1}{q}}\|f\|_{L^{q'}((0,T), L^{p'}(\real^d))}.
\end{align*}
\end{proof}

\noindent We set 
$$X_p(\real^d)=L^p(\real^d)\cap \dot{H}^1(\real^d).$$

\begin{nota}
From the Gagliardo-Nirenberg inequality, we have $H^1(\real^d)\hookrightarrow X_p(\real^d)$.
\end{nota}

\begin{prop}
The Schrödinger group $\{S(t)\}_{t\in\real}$ over $H^1(\real^d)$ defines, by continuous extension, a one-parameter continuous group on $X_p(\real^d)$.
\end{prop}
\begin{proof}
Given any $\phi\in \dot{H}^1(\real^d)$, we have $\|S(t)\nabla\phi\|_2=\|\nabla \phi\|_2$. Together with Proposition \ref{prop:estimativasLp}, this implies that
$$
\|S(t)\phi\|_{X_p}\lesssim (1+|t|^{1-\frac{2}{k}})^{1/2}\|\phi\|_{X_p}, \ t\in\real.
$$
Therefore, for each fixed $t\in\real$, $S(t)$ may be extended continuosly to $X_p$. By density, it follows easily that $S(t+s)=S(t)S(s)$, $t,s\in\real$, and $S(0)=I$ on $X_p$. Finally, we prove continuity at $t=0$: given $\phi\in X_p(\real^d)$ and $\epsilon>0$, take $\phi_\epsilon\in H^1(\real^d)$ such that 
$$
\|\phi_\epsilon-\phi\|_{X_p}<\epsilon.
$$
Then
\begin{align*}
\limsup_{t\to0}\|S(t)\phi - \phi\|_{X_p}&\le \limsup_{t\to0}\left(\|S(t)(\phi-\phi_\epsilon)\|_{X_p} + \|S(t)\phi_\epsilon-\phi_\epsilon\|_{X_p} + \|\phi_\epsilon-\phi\|_{X_p}\right)\\&\lesssim \limsup_{t\to0}\left((1+|t|^{1-\frac{2}{k}})^{1/2}\|\phi-\phi_\epsilon\|_{X_p} + \|S(t)\phi_\epsilon-\phi_\epsilon\|_{H^1} \right)\lesssim \epsilon.
\end{align*}
\end{proof}

\begin{nota}
Fix $d=1$. Using the same ideas, one may easily observe that the Schrödinger group is well-defined on the Zhidkov space
$$
X^2(\real)=\{u\in L^\infty(\real): \nabla u\in H^1(\real)\}.
$$
Indeed, for any $2\le p\le \infty$ a direct integration of the equation gives
$$
\frac{d}{dt}\|u(t)\|_p \le \|\Delta u(t)\|_p.
$$
Hence, choosing $k$ so that $(k,p)$ is an admissible pair,
$$
\|u(t)\|_p \le \|u_0\|_p + \int_0^t \|\Delta u(s)\|_p ds \le \|u_0\|_p + Ct^{1-\frac{1}{k}}\|\Delta u\|_{L^k((0,t), L^p)}\le\|u_0\|_p + Ct^{1-\frac{1}{k}}\|\Delta u_0\|_{2},
$$
where $C$ is a constant independent on $p$ (this comes from the fact that such a constant may be obtained via the interpolation between $L^{\infty}_t L^2_x$ and $L^{4}_tL^\infty_x$). Then, taking the limit $p\to \infty$, we obtain
$$
\|u(t)\|_\infty\lesssim \|u_0\|_\infty + t^{\frac{3}{4}}\|\Delta u_0\|_2, \quad t>0,\ u_0\in H^2(\real).
$$
For higher dimensions, a similar procedure may be applied, at the expense of some derivatives (one must use Sobolev injection to control $L^p$, with $p$ large). As one might expect, this argument does not provide the best possible estimate: in \cite{gallo}, one may see that
$$
\|u(t)\|_\infty \lesssim (1+t^{\frac{1}{4}})\left(\|u_0\|_\infty + \|\nabla u_0\|_2\right),\quad t>0, u_0\in H^1(\real).
$$
\end{nota}

\begin{nota}
One may ask if the required regularity is optimal: can we define the Schrödinger group on $X_p^s(\real^d):=\dot{H}^s(\real^d)\cap L^p(\real^d)$? What is the optimal $s$? Taking into consideration the previous remark, we conjecture that it should be possible to lower the regularity assumption. This entails a deeper analysis of the Schrödinger group, as it was done in \cite{gallo}.
\end{nota}

\section{Local well-posedness for $p\le 2\sigma+2$}

In order to clarify what do we mean by a solution of (NLS), we give the following
\begin{defi}[Solution over $X_p(\real^d)$]
Given $u_0\in X_p(\real^d)$, we say that $u\in C([0,T], X_p(\real^d))$ is a solution of (NLS) with initial data $u_0$ if the Duhamel formula is valid:
$$
u(t)=S(t)u_0 + i\lambda \int_0^t S(t-s)|u(s)|^\sigma u(s)ds,\quad t\in [0,T].
$$
\end{defi}
\noindent Throughout this section, let $(\gamma,\rho)$ and $(q,r)$ be admissible pairs such that
\begin{equation}\label{defipares}
r=(\sigma+1)\rho'=\max\{\sigma+2, p\}.
\end{equation}
It is easy to check that such pairs are well-defined for $p\le 2\sigma+2$.
\begin{prop}[Uniqueness over $X_p(\real^d)$]
Suppose that $p\le 2\sigma+2$. Let $u_1, u_2\in C([0,T], X_p(\real^d))$ be two solutions of (NLS) with initial data $u_0\in X_p(\real^d)$. Then $u_1\equiv u_2$.
\end{prop}
\begin{proof}
Taking the difference between the Duhamel formula for $u_1$ and $u_2$,
$$
u_1(t)-u_2(t)=i\lambda\int_0^t S(t-s)\left(|u_1(s)|^\sigma u_1(s) - |u_2(s)|^\sigma u_2(s)\right)ds
$$
Then, for any interval $J=[0,t]\subset [0,T]$, since $X_p(\real^d)\hookrightarrow L^r(\real^d)$,
\begin{align*}
\|u_1-u_2\|_{L^q(J, L^r)} &\lesssim \||u_1|^\sigma u_1- |u_2|^\sigma u_2\|_{L^{\gamma'}(J, L^{\rho'})} \\&\lesssim \left\|(\|u_1\|^\sigma_{r} + \|u_2\|_r^\sigma)\|u_1-u_2\|_r \right\|_{L^{\gamma'}(J)} \\&\lesssim \left(\|u_1\|_{L^\infty([0,T], X_p(\real^d))} + \|u_2\|_{L^\infty([0,T], X_p(\real^d))}\right)\|u_1-u_2\|_{L^{\gamma'}(J, L^r)}\\&\lesssim C(T)\|u_1-u_2\|_{L^{\gamma'}(J, L^r)}
\end{align*}
The claimed result now follows from \cite[Lemma 4.2.2]{cazenave}.
\end{proof}
\begin{teo}[Local well-posedness on $X_p(\real^d)$, $p\le 2\sigma+2$]\label{teo:lwp}
Given $u_0\in X_p(\real^d)$, there exists $T=T(\|u_0\|_{X_p})>0$ and an unique solution 
$$u\in C([0,T), X_p(\real^d))\cap L^\gamma((0,T),\dot{W}^{1,\rho}(\real^d))\cap L^q((0,T), \dot{W}^{1,r}(\real^d))$$
of (NLS) with initial data $u_0$. One has
\begin{equation}\label{smoothing}
u-S(\cdot)u_0 \in C([0,T]. L^2(\real^d))\cap L^q((0,T), L^{r}(\real^d)) \cap L^\gamma((0,T), L^{\rho}(\real^d)) .
\end{equation}
Moreover, the solution depends continuously on the initial data and may be extended in an unique way to a maximal time interval $[0,T^*(u_0))$. If $T^*(u_0)<\infty$, then
$$
\lim_{t\to T^*(u_0)} \|u(t)\|_{X_p} = +\infty.
$$
\end{teo}

\begin{nota}
The property \eqref{smoothing} is a type of nonlinear "smoothing" effect: the integral term in Duhamel's formula turns out to have more integrability than the solution itself (a similar property was seen in \cite{gerard}). This insight allows the use of Strichartz estimates at the zero derivatives level. Without this possibility, one would be restricted to the estimate \eqref{strichartzlocal} and the possible ranges of $\sigma$ and $p$ would be significantly smaller.
\end{nota}

\begin{proof}
\textit{Step 1.} Define
$$
\mathcal{S}_0=L^\infty((0,T), L^2) \cap L^q((0,T), L^r) \cap L^\gamma((0,T), L^\rho).
$$
and
$$
\mathcal{S}_1=L^\infty((0,T), H^1) \cap L^q((0,T), W^{1,r}) \cap L^\gamma((0,T), W^{1,\rho}).
$$
Consider the space
\begin{align*}
\mathcal{E}=\Big\{& u\in L^\infty((0,T), X_p)\cap L^\gamma((0,T), \dot{W}^{1,\rho})\cap L^q((0,T), \dot{W}^{1,r}):\\ & \vvvert u\vvvert:= \|u\|_{L^\infty((0,T), L_p)} + \|u - S(\cdot)u_0\|_{\mathcal{S}_1} \le M\Big\}.
\end{align*}
endowed with the distance
$$
d(u,v)= \| u - v \|_{\mathcal{S}_0}.
$$
It is not hard to check that $(\mathcal{E}, d)$ is a complete metric space: indeed, if $\{u_n\}_{n\in\nat}$ is a Cauchy sequence in $\mathcal{E}$, then $\{u_n-S(\cdot)u_0\}_{n\in\nat}$ is a Cauchy sequence in $\mathcal{S}_0$.
Then there exists $u\in \mathcal{D}'([0,T]\times\real^d)$ such that $u_n-S(\cdot)u_0\to u-S(\cdot)u_0$ in $\mathcal{S}_0$. By \cite[Theorem 1.2.5]{cazenave}, this convergence implies that 
$$u-S(\cdot)u_0\in \mathcal{S}_1,\quad \|u-S(\cdot)u_0\|_{\mathcal{S}_1} \le \liminf \|u_n-S(\cdot)u_0\|_{\mathcal{S}_1}
$$
Finally, it follows from the Gagliardo-Nirenberg inequality that, for some $0<\theta<1$,
$$
\|u_n-u\|_{L^\infty((0,T), L^p)}\lesssim \|u_n-u\|_{L^\infty((0,T), L^2)}^{1-\theta}\|\nabla u_n-\nabla u\|_{L^\infty((0,T), L^2)}^\theta \to 0
$$
and so $u_n\to u$ in $L^\infty((0,T), L^p)$.

\noindent \textit{Step 2.} Define, for any $u\in \mathcal{E}$,
$$
(\Phi u)(t)= S(t)u_0 +i\lambda \int_0^t S(t-s)|u(s)|^\sigma u(s) ds, \quad 0<t<T.
$$
It follows from the definition of $r$ (see \eqref{defipares}) that $X_p(\real^d)\hookrightarrow L^r(\real^d)$. Then
\begin{align*}
\|\Phi u - S(\cdot)u_0\|_{\mathcal{S}_1}&\lesssim \||u|^\sigma u\|_{L^{\gamma'}((0,T), W^{1,\rho'})} \\&\lesssim \left\| \|u\|_{r}^{\sigma}( \|u\|_{r} + \|\nabla u\|_{r})\right\|_{L^{\gamma'}(0,T)}\\&\lesssim  \left\| \|u\|_{X_p}^{\sigma}( \|u\|_{X_p} + \|\nabla (u -S(\cdot)u_0)\|_{r} + \|S(\cdot)\nabla u_0\|_{r}\right\|_{L^{\gamma'}(0,T)}\\& \lesssim T^{\frac{1}{\gamma'}}\|u\|_{L^\infty((0,T), X_p)}^{\sigma+1} + T^{\frac{1}{\gamma'}-\frac{1}{q}}\|u\|_{L^\infty((0,T), X_p)}^{\sigma}\|\nabla (u -S(\cdot)u_0)\|_{L^{q}((0,T), L^{r})} \\&+ T^{\frac{1}{\gamma'}-\frac{1}{q}}\|u\|_{L^\infty((0,T), X_p)}^{\sigma}\|S(\cdot)\nabla u_0\|_{L^{q}((0,T), L^{r})}\\&
\lesssim \left(T^{\frac{1}{q'}} + T^{\frac{1}{q'}-\frac{1}{q}}\right)(M^{\sigma+1} + \|u_0\|_{X_p}^{\sigma+1}).
\end{align*}
It follows that, for $M \sim 2\|u_0\|_{X_p}$ and $T$ sufficiently small, we have $\Phi:\mathcal{E}\mapsto \mathcal{E}$.

\noindent \textit{Step 3.} Now we show a contraction estimate: given $u,v\in\mathcal{E}$, 
\begin{align*}
d(\Phi(u), \Phi(v))&\lesssim \| |u|^\sigma u - |v|^\sigma v\|_{L^{\gamma'}(L^{\rho'})} \\&\lesssim \left\| (\|u\|_{r}^\sigma + \|v\|_{r}^\sigma)\|u-v\|_{r} \right\|_{L^{\gamma'}(0,T)} \\&\lesssim T^{\frac{1}{\gamma'}-\frac{1}{q}}\left(\|u\|_{L^\infty((0,T), X_p(\real^d)}^\sigma + \|u\|_{L^\infty((0,T), X_p(\real^d)}^\sigma\right) \|u-v\|_{L^q((0,T), L^{r})} \\&\lesssim T^{\frac{1}{q'}-\frac{1}{q}}\left(M^\sigma + \|u_0\|_{X_p}^\sigma\right)d(u,v).
\end{align*}

Therefore, for $T=T(\|u_0\|_{X_p})$ small enough, the mapping $\Phi:\mathcal{E}\mapsto \mathcal{E}$ is a strict contraction and so, by Banach's fixed point theorem, $\Phi$ has a unique fixed point over $\mathcal{E}$. This gives the local existence of a solution $u\in C([0,T], X_p(\real^d))$ of (NLS) with initial data $u_0$. From the uniqueness result, such a solution can then be extended to a maximal interval of existence $(0,T^*(u_0))$. If such an interval is bounded, then necessarily one has $\|u(t)\|_{X_p}\to \infty$ as $t\to T^*(u_0)$. The continuous dependence on the initial data follows as in the $H^1$ case (see, for example, the proof of \cite[Theorem 4.4.1]{cazenave})
\end{proof}

\begin{nota}\label{nota:restricaop}
The condition $p\le 2\sigma+2$ is necessary for one to use Strichartz estimates with no derivatives. Indeed, when one applies Strichartz to the integral term of the Duhamel formula, one has
$$
\left\|\int_0^\cdot S(\cdot - s)|u(s)|^{\sigma}u(s)ds\right\|_{L^{q}((0,T), L^r)}\lesssim \|u\|_{L^{\gamma'(\sigma+1)}((0,T), L^{\rho'(\sigma+1)})}^{\sigma+1}
$$
for any admissible pairs $(q,r)$ and $(\gamma, \rho)$. Since the solution $u$ only lies on spaces with spatial integrability larger or equal than $p$, one must have $p\le \rho'(\sigma+1) \le 2\sigma+2$ (because $\rho\ge 2$).
\end{nota}

\begin{prop}[Persistence of integrability]\label{prop:persist1}
Suppose that $\tilde{p}<p\le 2\sigma+2$. Given $u_0\in X_{\tilde{p}}(\real^d)$, consider the $X_p(\real^d)$-solution $u\in C([0,T^*(u_0)),X_{p})$ of (NLS) with initial data $u_0$. Then $u\in C([0,T^*(u_0)),X_{\tilde{p}})$.
\end{prop}

\begin{proof}
By the local well-posedness result over $X_{\tilde{p}}(\real^d)$ and by the uniqueness over $X_p(\real^d)$, there exists a time $T_0>0$ such that $u\in C([0,T_0], X_{\tilde{p}}(\real^d))$. Thus the statement of the proposition is equivalent to saying that $u$ does not blow-up in $X_{\tilde{p}}(\real^d)$ at a time $T_0<T<T^*(u_0)$. Since $u$ is bounded in $X_p$ over $[0,T]$, it follows from the local existence theorem that
$$
\|u-S(\cdot)u_0\|_{L^\infty((0,T), H^1)} <\infty.
$$
Then, for any $0<t<T$,
\begin{align*}
\|u\|_{L^\infty((0,t), X_{\tilde{p}})} &\lesssim \|S(\cdot)u_0\|_{L^\infty((0,t), X_{\tilde{p}})} + \|u-S(\cdot)u_0\|_{L^\infty((0,t), X_{\tilde{p}})} \\&\lesssim \|u_0\|_{X_{\tilde{p}}} + \|u-S(\cdot)u_0\|_{L^\infty((0,t), H^1)} <\infty,
\end{align*}
which implies that $u$ does not blow-up at time $t=T$.
\end{proof}

\begin{prop}[Conservation of energy]
Suppose that $p\le \sigma+2$. Given $u_0\in X_p(\real^d)$, the corresponding solution $u$ of (NLS) with initial data $u_0$ satisfies
$$
E(u(t))=E(u_0):=\frac{1}{2}\|\nabla u_0\|_2^2 - \frac{\lambda}{\sigma+2}\|u_0\|_{\sigma+2}^{\sigma+2}, \ 0<t<T(u_0).
$$
Consequently, if $\lambda<0$, then $T^*(u_0)=\infty$. Moreover, if $\lambda>0$ and $T^*(u_0)<\infty$, then
\begin{equation}\label{limites}
\lim_{t\to T^*(u_0)} \|\nabla u(t)\|_2 = \lim_{t\to T^*(u_0)} \|u(t)\|_{\sigma+2} = \infty. 
\end{equation}
\end{prop}
\begin{proof}
Since the conservation law is valid for $u_0\in H^1(\real^d)$, through a regularization argument, the same is true for any $u_0\in X_p(\real^d)$. If $\lambda<0$, one has
$$
\|u(t)\|_{X_{\sigma+2}(\real^d)} \lesssim E(u_0), \quad 0<t<T^*(u_0).
$$
By the blow-up alternative, this implies that $u$, as a $X_{\sigma+2}(\real^d)$ solution, is globally defined. By persistence of integrability, this implies that $u$ is global in $X_p(\real^d)$. If $\lambda>0$, suppose by contradiction that \eqref{limites} is not true. Then, by conservation of energy, $u$ is bounded in $X_{\sigma+2}(\real^d)$ and therefore it is globally defined (as an $X_{\sigma+2}(\real^d)$ solution, but also as an $X_p(\real^d)$ solution, by persistence of integrability). 
\end{proof}

\begin{prop}\label{prop:globalsubcritico}
Fix $\lambda<0$. If
\begin{equation}
2\sigma^2 + (d+2)\sigma \le 4,
\end{equation}
then, for any $u_0\in X_{\sigma+2}(\real^d)$, the corresponding solution $u$ of (NLS) is globally defined.
\end{prop}
\begin{nota}
Notice that the condition on $\sigma$ implies that $\sigma<\min\{\sqrt{2}, 4/(d+2)\}<4/d$. 
\end{nota}
\begin{proof}
By contradiction, assume that $u$ blows-up at time $t=T$. The previous proposition then implies that
$$
\lim_{t\to T} \|\nabla u(t)\|_2 =\infty.
$$
The first step is to obtain a corrected mass conservation estimate: indeed, by direct integration of the equation,
\begin{align*}
\frac{1}{2}\frac{d}{dt} \|u(t)-S(t)u_0\|_2^2 &= \parteim \int |u(t)|^\sigma u(t) \overline{(u(t)-S(t)u_0)} = - \parteim \int |u(t)|^\sigma u(t) \overline{S(t)u_0} \\&\lesssim \|u(t)\|_{\sigma+2}^{\sigma+1}\|S(t)u_0\|_{\sigma+2}.
\end{align*}
Integrating on $(0,t)$,
\begin{align*}
\|u(t)-S(t)u_0\|_2^2 &\lesssim \int_0^t \|u(s)\|_{\sigma+2}^{\sigma+1}\|S(s)u_0\|_{\sigma+2} ds \lesssim \|S(\cdot)u_0\|_{L^\infty((0,T), L^{\sigma+2})} \int_0^t \|u(s)\|_{\sigma+2}^{\sigma+1} ds \\&\lesssim \int_0^t \|u(s)\|_{\sigma+2}^{\sigma+1} ds
\end{align*}
All of these formal computations can be justified by a suitable regularization and approximation argument. The next step is to use the conservation of energy and the Gagliardo-Nirenberg inequality to obtain a bound on $\|\nabla u(t)\|_2$.
\begin{align*}
\frac{1}{2}\|\nabla u(t)\|_2^2 &= E(u_0) + \frac{1}{\sigma+2}\|u(t)\|_{\sigma+2}^{\sigma+2}  \\&\lesssim 1 + \|u(t)-S(t)u_0\|_{\sigma+2}^{\sigma+2} + \|S(t)u_0\|_{\sigma+2}^{\sigma+2} \\&\lesssim 1+ \|\nabla(u(t)- S(t)u_0)\|_2^{\frac{d\sigma}{2}} \|u(t)-S(t)u_0\|_2^{\frac{4-(d-2)\sigma}{2}}\\& \lesssim  1+ \|\nabla(u(t)- S(t)u_0)\|_2^{\frac{d\sigma}{2}} \left(\int_0^t \|u(s)\|_{\sigma+2}^{\sigma+1}ds\right)^{\frac{4-(d-2)\sigma}{4}}
\end{align*}
For $t$ close to $T$, $\|\nabla(u(t)- S(t)u_0)\|_2 \sim \|\nabla u(t)\|_2$ and, by conservation of energy, 
$$
\|u(t)\|_{\sigma+2}^{\sigma+1}\lesssim \|\nabla u(t)\|_2^{\frac{2\sigma+2}{\sigma+2}}
$$
Thus
\begin{align*}
\|\nabla u(t)\|_2^{\frac{4-d\sigma}{2}} \lesssim 1 + \left(\int_0^t \|\nabla u(s)\|_2^{\frac{2\sigma+2}{\sigma+2}}ds\right)^{\frac{4-(d-2)\sigma}{4}}.
\end{align*}
which, together with the condition on $\sigma$, implies that
$$
g(t):=\|\nabla u(t)\|_2^{\frac{2\sigma+2}{\sigma+2}}\le \|\nabla u(t)\|_2^{\frac{2(4-d\sigma)}{4-(d-2)\sigma}} \lesssim 1 + \int_0^t  \|\nabla u(s)\|_2^{\frac{2\sigma+2}{\sigma+2}}ds \lesssim 1 + \int_0^t g(s) ds.
$$
The desired contradiction now follows from a standard application of Gronwall's lemma.
\end{proof}

\begin{nota}[Global existence in the focusing $L^2$-subcritical regime]
As it is well-known, the global existence in $H^1(\real^d)$ for $\sigma<4/d$ follows easily from the conservation of mass and energy and from the Gagliardo-Nirenberg inequality. In Proposition \ref{prop:globalsubcritico}, we managed to perform a similar argument by using the corrected mass 
$$
M(t)=\|u(t)-S(t)u_0\|_2^2.
$$
However, the range of exponents for which the result is valid still leaves much to be desired. We are left with some questions: Is there another choice for "corrected mass" that allows a larger range of exponents? Is it possible that the large tails of the initial data contribute to blow-up behaviour?
\end{nota}

\begin{nota}[Blow-up in the $L^2$-(super)critical regime]
One may ask whether the known blow-up results for $\sigma\ge 4/d$ can be extended to initial data in $X_p(\real^d)$ which do not lie in $L^2(\real^d)$. First of all, notice that
$$
X_p(\real^d)\cap L^2(|x|^2dx) \hookrightarrow L^2(\real^d).
$$
Thus, in order to obtain blow-up outside $L^2$, one must first show blow-up in $H^1$ without the finite variance assumption. This is an open problem, which has been solved in \cite{ogawatsutsumi} under radial hypothesis and relying heavily on the conservation of mass (which is unavailable on $X_p(\real^d)$). For the nonradial case, recent works (see, for example, \cite{roudenko}) only manage to prove unboundedness of solutions of negative energy. The problem of blow-up solutions strictly in $X_p(\real^d)$ is an even harder problem, requiring a better control on the tails of the solution.
\end{nota}

\begin{nota}[Scaling invariance]
It is useful to understand how scalings affect the $X_p(\real^d)$ norm: recalling that the (NLS) is invariant under the scaling $u_\lambda(t,x)=\lambda^{2/\sigma}u(\lambda^2 t, \lambda x)$, we have
$$
\|u_\lambda (t)\|_{p} = \lambda^{\frac{2}{\sigma}-\frac{d}{p}} \|u(\lambda^2 t)\|_p, \quad \|\nabla u_\lambda(t)\|_2 = \lambda^{\frac{2}{\sigma}+1-\frac{d}{2}}\|\nabla u(\lambda^2 t)\|_2.
$$
Thus the (NLS) is $X_p(\real^d)$-subcritical for $\sigma<2p/d$. In this situation,  global existence for small data is equivalent to global existence for any data. Recall, however, that, for $\sigma\ge 4/d$, existence of blow-up phenomena is known for special initial data in $H^1(\real^d)\hookrightarrow X_p(\real^d)$. Therefore, it is impossible to obtain a global existence result for small data for $4/d\le \sigma < 2p/d$. Notice that in the energy case $p=\sigma+2$, one has $\sigma < 2p/d$ for any $\sigma+2<2^*$.
\end{nota}

\begin{nota}[Global existence for small data]
The main obstacle in proving global existence for small data turns out to be the linear part of the Duhamel formula $S(t)u_0$, since there isn't, to our knowledge, a way to bound uniformly this term over $X_p(\real^d)$. The other possibility is to leave the linear term with a space-time norm: indeed, for some powers $\sigma>2/d$, it is well-known that, if $u_0\in H^1(\real^d)$ is such that
$$
\|S(\cdot)u_0\|_{L^a((0,\infty), L^{\sigma+2}(\real^d))}\mbox{ is small}, \quad a=\frac{2\sigma(\sigma+2)}{4-\sigma(d-2)},
$$
then the corresponding solution of (NLS) is globally defined (see \cite{cazenaveweissler}). It is not hard to check that the result can be extended to $u_0\in X_{\sigma+2}(\real^d)$.
\end{nota}

\section{Local well-posedness for $p>2\sigma+2$}

As it was observed in Remark \ref{nota:restricaop}, the condition $p\le 2\sigma+2$ was necessary in order to use Strichartz estimates with no loss in regularity. For $p>2\sigma+2$, in order to estimate $L^\infty_tL^p_x$, one must turn to estimate \eqref{eq:naohomogenea}, which has a loss of one derivative. Therefore the distance one defines for the fixed-point argument must include norms with derivatives. This implies the need of a local Lipschitz condition 
$$
||u|^\sigma \nabla u - |v|^\sigma \nabla v|\lesssim C(|u|, |v|, |\nabla u|, |\nabla v|)\left(|u-v| + |\nabla (u-v)|\right),
$$
which we can only accomplish for $\sigma \ge 1$.

Because of the restriction $\sigma\ge 1$, one must have $4<p<2^*$, which excludes any dimension greater than three. For $d=3$, it turns out that no range of $p>2\sigma+2$ can be considered. Indeed, if one uses \eqref{eq:naohomogenea} with $f=|u|^\sigma u$,
$$
\left\| \int_0^\cdot S(\cdot-s)|u(s)|^\sigma u(s)ds\right\|_{L^\infty((0,T), L^p)}\lesssim \|u\|_{L^{2\sigma+2}((0,T), L^{p(\sigma+1)})}^{\sigma+1} + \|\nabla (|u|^\sigma u)\|_{L^{q'}((0,T), L^{r'})}.
$$ 
We focus on the first norm on the right hand side. To control such a term, either $X_p\hookrightarrow L^{p(\sigma+1)}_x$ and
$$
\|u\|_{L^{2\sigma+2}((0,T), L^{p(\sigma+1)}} \lesssim T^{\frac{1}{2\sigma+2}}\|u\|_{L^\infty((0,T), X_p)},
$$
or, setting $r\ge 2$ so that
$$
1-\frac{3}{r} = -\frac{3}{p(\sigma+1)},
$$
one estimates
$$
\|u\|_{L^{2\sigma+2}((0,T), L^{p(\sigma+1)})} \lesssim \| \nabla u\|_{L^{2\sigma+2}((0,T), L^r)} \lesssim T^{\frac{1}{2\sigma+2}-\frac{1}{q}}\| \nabla u\|_{L^{q}((0,T), L^r)}.
$$
In the first case, one needs $8<p(\sigma+1)<2^*=6$. In the second, one must impose $2\sigma+2<q$. A simple computation yields $p(3\sigma+1)<6$, which is again impossible, since $p(3\sigma+1)>16$.


\begin{teo}[Local well-posedness on $X_p(\real^d)$ for $d=1,2$]
Given $u_0\in X_p(\real^d)$, there exists $T=T(\|u_0\|_{X_p})>0$ and an unique solution $$u\in C([0,T], X_p(\real^d))$$
of (NLS) with initial data $u_0$. The solution depends continuously on the initial data and may be extended uniquely to a maximal interval $[0,T^*(u_0))$. If $T^*(u_0)<\infty$, then
$$
\lim_{t\to T^*(u_0)} \|u(t)\|_{X_p} = +\infty.
$$
\end{teo}

\begin{proof}
Consider the space
\begin{align*}
\mathcal{E}=\Big\{ u\in L^\infty((0,T), X_p): \vvvert u\vvvert:= \|u\|_{L^\infty((0,T), X_p)} \le M \Big\}.
\end{align*}
endowed with the natural distance
$$
d(u,v)= \vvvert u-v\vvvert.
$$
The space $(\mathcal{E}, d)$ is clearly a complete metric space. If $u,v\in \mathcal{E}$, then
\begin{align*}
\||u|^\sigma u - |v|^\sigma v\|_{L^2((0,t), L^p)}^2\lesssim \int_0^t \left(\|u\|_{p(\sigma+1)}^{2\sigma} + \|v\|_{p(\sigma+1)}^{2\sigma}\right) \|u-v\|_{p(\sigma+1)}^2 ds
\end{align*}
Since $X_p(\real^d)\hookrightarrow L^{p(\sigma+1)}(\real^d)$,
\begin{equation}\label{eq:lwp21}
\||u|^\sigma u - |v|^\sigma v\|_{L^2((0,t), L^p)}^2\lesssim  T\left(\|u\|_{L^\infty((0,t). X_p)}^{2\sigma} + \|v\|_{L^\infty((0,t). L^p)}^{2\sigma}\right) \|u-v\|_{L^\infty((0,t). X_p)}^{2}.
\end{equation}
Choose an admissible pair $(\gamma,\rho)$ with $\rho$ sufficiently close to 2. We have
\begin{align*}
\|\nabla \left(|u|^\sigma u - |v|^\sigma v\right)\|_{L^{\gamma'}((0,T), L^{\rho'})} &\lesssim \left\|\left(|u|^{\sigma-1} + |v|^{\sigma-1}\right)(|u-v||\nabla v| + |v||\nabla (u-v)|) \right\|_{L^{\gamma'}((0,T), L^{\rho'})}.
\end{align*}
As an example, we treat the term $|u|^{\sigma-1}|u-v||\nabla v|$:
\begin{align*}
\||u|^{\sigma-1}|u-v||\nabla v|\|_{\rho'}&\lesssim \|u\|_{\frac{2\sigma \rho'}{2-\rho'}}^{\sigma-1} \|u-v\|_{\frac{2\sigma \rho'}{2-\rho'}}\|\nabla v\|_{2}\\&\lesssim \|u\|_{X_p}^\sigma \|u-v\|_{X_p},
\end{align*}
Therefore
\begin{align}
\|\nabla \left(|u|^\sigma u - |v|^\sigma v\right)\|_{L^{\gamma'}((0,T), L^{\rho'})}&\lesssim
T^{\frac{1}{\gamma'}} \left(\|u\|_{L^\infty((0,T), X_p)}^\sigma + \|v\|_{L^\infty((0,T), X_p)}^\sigma\right)\|u-v\|_{L^\infty((0,T), X_p)} \nonumber\\&\lesssim T^{\frac{1}{\gamma'}} M^\sigma d(u,v).\label{eq:lwp22}
\end{align}
For $u\in\mathcal{E}$, define
$$
\Phi(u)(t)=S(t)u_0 +i\lambda \int_0^t S(t-s)|u(s)|^\sigma u(s)ds, \quad 0\le t\le T.
$$
The estimates \eqref{eq:lwp21} and \eqref{eq:lwp22}, together with \eqref{eq:naohomogenea} and Strichartz's estimates then imply that
\begin{align*}
\vvvert \Phi(u)\vvvert \lesssim&  \|u_0\|_{X_p}  + \left\|\int_0^\cdot S(\cdot-s)|u(s)|^\sigma u(s)ds\right\|_{L^\infty((0,T), L^p)} + \left\|\int_0^\cdot S(\cdot-s)|u(s)|^\sigma u(s)ds\right\|_{L^\infty((0,T), \dot{H}^1)}\\&\lesssim \|u_0\|_{X_p} + \left(\||u|^\sigma u\|_{L^2((0,T); L^p)} + \|\nabla (|u|^\sigma u)\|_{L^{\gamma'}((0,T); L^{\rho'})}\right)\\&\lesssim \|u_0\|_{X_p} + \left(T^{\frac{1}{2}} + T^{\frac{1}{\gamma'}}\right) M^{\sigma+1}
\end{align*}
and
\begin{align}
d(\Phi(u),\Phi(v)) &\lesssim \left(\||u|^\sigma u - |v|^\sigma v\|_{L^2((0,T); L^p)} + \|\nabla (|u|^\sigma u)- \nabla( |v|^\sigma v)\|_{L^{\gamma'}((0,T); L^{\rho'})}\right)\nonumber\\ &\lesssim \left(T^{\frac{1}{2}} + T^{\frac{1}{\gamma'}}\right)M^{\sigma}d(u,v).\label{eq:contracao}
\end{align}
Choosing $M\sim 2\|u_0\|_{X_p}$, for $T=T(\|u_0\|_{X_p})$ small enough, it follows that $\Phi:\mathcal{E}\mapsto \mathcal{E}$ is a strict contraction. Banach's fixed point theorem now implies that $\Phi$ has a unique fixed point over $\mathcal{E}$, which is the unique solution $u$ of (NLS) with initial data $u_0$ on the interval $(0,T)$. This solution may then be extended uniquely to a maximal interval of existence $(0,T(u_0))$. The blow-up alternative follows by a standard continuation argument. Finally, if $u,v$ are two solutions with initial data $u_0, v_0\in X_p(\real^d)$, as in \eqref{eq:contracao}, one has
\begin{align*}
d(u, v) = d(\Phi(u), \Phi(v))&\lesssim \|u_0-v_0\|_{X_p} + \left(T^{\frac{1}{2}} + T^{\frac{1}{\gamma'}}\right)M^{\sigma}d(u,v) \\&\lesssim \|u_0-v_0\|_{X_p} + \left(T^{\frac{1}{2}} + T^{\frac{1}{\gamma'}}\right)\left(\max\{\|u_0\|_{X_p}, \|v_0\|_{X_p}\}\right)^{\sigma}d(u,v) 
\end{align*}
Thus, for $T_0=T_0(\|u_0\|_{X_p}, \|v_0\|_{X_p})$ small,
$$
d(u,v)\lesssim  \|u_0-v_0\|_{X_p},
$$
and continuous dependence follows.
\end{proof}

\begin{prop}[Persistence of integrability]
Fix $d=1,2$ and $p>\tilde{p}$. Given $u_0\in X_{\tilde{p}}(\real^d)$, consider the $X_p(\real^d)$-solution $u\in C([0,T^*(u_0)),X_{p})$ of (NLS) with initial data $u_0$. Then $u\in C([0,T^*(u_0)),X_{\tilde{p}})$.
\end{prop}
\begin{proof}
As in the proof of Proposition \ref{prop:persist1}, given $T<T^*(u_0)$, one must prove that the $L^{\tilde{p}}$ norm of $u$ is bounded over $(0,T)$. Applying \eqref{eq:naohomogenea} to the Duhamel formula of $u$,
\begin{align*}
\|u\|_{L^\infty((0,T), L^{\tilde{p}})} \lesssim \|u_0\|_{X_{\tilde{p}}} + \||u|^\sigma u\|_{L^2((0,T), L^{\tilde{p}})} + \||u|^\sigma |\nabla u|\|_{L^{\gamma'}((0,T), L^{\rho'})},
\end{align*}
for any admissible pair $(\gamma,\rho)$. The penultimate term is treated using the injection $X_p(\real^d)\hookrightarrow L^{\tilde{p}(\sigma+1)}$:
$$
\||u|^\sigma u\|_{L^2((0,T), L^{\tilde{p}})} = \|u\|_{L^{2\sigma+2}((0,T), L^{\tilde{p}(\sigma+1)})}^{\sigma+1} \lesssim T^{\frac{1}{2}}\|u\|_{L^\infty((0,T), X_p(\real^d)}^{\sigma+1}<\infty.
$$
Choose $\rho$ sufficiently close to 2 so that $X_p(\real^d)\hookrightarrow L^{\frac{2\sigma \rho'}{2-\rho'}}(\real^d)$. Then
$$
\||u|^\sigma |\nabla u|\|_{L^{\gamma'}((0,T), L^{\rho'})} \lesssim \left\|\|u\|_{\frac{2\sigma\rho'}{2-\rho'}}^\sigma\|\nabla u\|_2 \right\|_{L^{\gamma'}(0,T)}\lesssim T^{\frac{1}{\gamma'}} \|u\|_{L^\infty((0,T), X_p)}^{\sigma+1}<\infty.
$$
Therefore $\|u\|_{L^\infty((0,T), L^{\tilde{p}})}$ is finite and the proof is finished.
\end{proof}

\section{Further comments}

In light of the results we have proven, we highlight some new questions that have risen:
\begin{enumerate}
\item Local well-posedness: In dimensions $d\ge 3$, the local well-posedness in the case $p>2\sigma+2$ remains open. Is this optimal? As we have arqued in Remark \ref{nota:restricaop}, this case requires new estimates for the Schrödinger group.
\item Global well-posedness: this problem is completely open for $p>\sigma+2$. Even if the energy is well-defined, there are still several cases where global well-posedness (even for small data) remains unanswered.
\item New blow-up behaviour: in the opposite perspective, is it possible to exhibit new blow-up phenomena? This would be especially interesting either for the defocusing case or for the $L^2$-subcritical case, where blow-up behaviour in $H^1$ is impossible. 
\item Stability of ground-states: in the $H^1$ framework, the work of \cite{cazenavelions} has shown that the ground-states are orbitally stable under $H^1$ perturbations. Does the result still hold if we consider $X_p$ perturbations?
\end{enumerate}

\section{Acknowledgements}
The author was partially suported by Fundação para a Ciência e Tecnologia, through the grants UID/MAT/04561/2013 and SFRH/BD/96399/2013.

\bibliography{Biblioteca}
\bibliographystyle{plain}

\small
\noindent \textsc{Sim\~ao Correia}\\
CMAF-CIO and FCUL \\
\noindent Campo Grande, Edif\'icio C6, Piso 2, 1749-016 Lisboa (Portugal)\\
\verb"sfcorreia@fc.ul.pt"\\

\end{document}